\documentclass[
10pt, 
a4paper, 
oneside, 
headinclude,footinclude, 
BCOR = 5mm, 
]{scrartcl}

\usepackage[
nochapters, 
pdfspacing, 
dottedtoc 
]{classicthesis} 

\usepackage[left=1in,right=1in,top=1in,bottom=2in]{geometry}

\usepackage[T1]{fontenc} 

\usepackage[utf8]{inputenc} 

\usepackage{graphicx} 
\graphicspath{{Figures/}} 

\usepackage{enumitem} 

\usepackage{lipsum} 

\usepackage{subfig} 

\usepackage{amsmath,amssymb,amsthm,tikz-cd, mathtools} 

\usepackage{varioref} 

\theoremstyle{definition} 
\newtheorem{definition}{Definition}

\theoremstyle{plain} 
\newtheorem{theorem}{Theorem}
\newtheorem*{theorem*}{Theorem}

\newtheorem*{corollary*}{Corollary}

\theoremstyle{remark} 

\newtheorem{remark}{Remark}
\newtheorem{example}{Example}

\hypersetup{
colorlinks=true, breaklinks=true, bookmarks=true,bookmarksnumbered,
urlcolor=webbrown, linkcolor=RoyalBlue, citecolor=webgreen, 
pdftitle={}, 
pdfauthor={\textcopyright}, 
pdfsubject={}, 
pdfkeywords={}, 
pdfcreator={pdfLaTeX}, 
pdfproducer={LaTeX with hyperref and ClassicThesis} 
}

\title{\normalfont{Supersingularity of Motives with Complex Multiplication and a Twisted Polarization}} 

\author{\spacedlowsmallcaps{Asvin G}\footnote{\textit{Department of Mathematics; University of Wisconsin, Madison.}}
\footnote{\textit{  email: gasvinseeker94@gmail.com}}} 

\date{} 

\begin{document}


\renewcommand{\sectionmark}[1]{\markright{\spacedlowsmallcaps{#1}}} 
\lehead{\mbox{\llap{\small\thepage\kern1em\color{halfgray} \vline}\color{halfgray}\hspace{0.5em}\rightmark\hfil}} 

\pagestyle{scrheadings} 


\maketitle 

\setcounter{tocdepth}{2} 



\begin{abstract}
Let $X$ be a smooth, projective variety over a finite field $\mathbb F_q$. We say that its crystalline cohomology $H^i_{\mathrm{crys}}(X)/W(\mathbb F_q)$ is supersingular if all the eigenvalues for the action of the Frobenius $\sigma_q$ on it are of the form $q^{i/2}\zeta$ for $\zeta$ a root of unity. In this note, we prove a criteria for supersingularity when the variety has a large automorphism group and a perfect bilinear pairing. This criteria unifies and extends many known results on the supersingularity of curves and varieties and in particular, applies to a large family of Artin-Schreier curves.
    
\end{abstract} 





\section{Introduction}

Informally, a smooth projective variety $X/\mathbb F_q$ is said to be supersingular if for each $i$, the Newton slopes of the motive $H^i(X)$ are all equal to $i/2$. There are many (equivalent) ways of making this definition precise and in the case of Abelian varieties (and $K3$ surfaces), it is a well studied notion. For the purposes of this article, we will use the slightly non standard definition:
\begin{definition}
A smooth, projective variety $X/\mathbb F_q$ is said to be supersingular in degree $i$ if the eigenvalues of the Frobenius $\sigma_q$ on the crsytalline cohomology $H^i_{\mathrm{crys}}(X)/W(\mathbb F_q)$ are of the form $q^{i/2}\zeta$ for $\zeta$ a root of unity. We say $X$ is supersingular if it is supersingular in every degree $i$.

\end{definition}

In this paper, we consider varieties $X$ with a large automorphism group $G$ so that $H^i_{\text{\'et}}(\overline{X},\mathbb Q_\ell)$ is a quotient of $\mathbb Q_\ell[G]$ as $G$ modules compatible with the Frobenius action ("complex multiplication") and an equivariant inner product
    \[H^i_{\text{\'et}}(\overline{X},\mathbb Q_\ell)\otimes H^i_{\text{\'et}}(\overline{X},\mathbb Q_\ell) \to \mathbb Q_\ell(-d).\]

\begin{theorem}
    If $H^i_{\text{\'et}}(\overline{X},\mathbb Q_\ell)$ is self-dual, i.e., for any character $\chi$ of $G$, the $\chi^{-1}$-isotypic eigenspace of $H^i_{\text{\'et}}(\overline{X},\mathbb Q_\ell)$ is in the Frobenius orbit of the $\chi$-isotypic eigenspace, then $X$ is supersingular in degree $i$. The converse is also true if $d=1$.
\end{theorem}
The proof is elementary and it generalizes and unifies many old results proven by disparate methods. Some examples of this theorem are as follows. In each case $n$ is a positive integer co-prime to the characteristic $p$. An interesting feature of the proof is that while supersingularity is invariant under extending the base field, our necessity criterion is not. Therefore, the minimal field of definition of $X$ becomes important for this direction.

\begin{enumerate}
    \item \textbf{Fermat varieties:} Hypersufaces of dimension $d$ defined by
    \[F_{n,d}: X_0^n + \dots + X_{d+1}^n = 0.\]
    In this case, much is known by the results of Shioda \cite{shioda1979fermat} and our methods offer an alternate proof of some of their results. For Fermat Curves in particular ($d=1$), we prove that $F_{n,1}$ is supersingular in characteristic $p$ if and only if there exists some $s\geq 1$ such that $p^s \equiv -1 \pmod{n}$. In higher dimensions, we prove that the condition is sufficient while \cite{shioda1979fermat} also proves necessity. This also implies the supersingularity of quotients such as the Hurwitz curves \cite{Dawson_2019} and superelliptic curves since they are quotients of Fermat curves.
    
    \item \textbf{Abelian varieties with CM by a cyclotomic field:} Let $A_n/\mathbb F_q$ be an abelian variety with $\mu_n \subset \mathrm{End}_{\overline{\mathbb F}_q}(A_n)$ and $\dim A_n \leq n$. For instance $A$ could be the reduction of an abelian variety in characteristic zero with CM by $\mathbb Q(\mu_n)$. In this case, we prove that $A_n$ is supersingular if there exists a $s \geq 1$ such that $q^s \equiv -1 \pmod{n}$. If moreover $q$ is a prime, i.e., the field of moduli\footnote{Over a finite field, the field of moduli coincides with the field of definition} is a prime field then the converse is also true. This is a special case of a known fact (Remark \ref{rmk: CM AV}) although the method of proof is completely different.
    
    \item \textbf{Artin-Schreier curves:} Let $C_{q,n}$ be the (smooth, projective models corresponding to the) degree $n$ cyclic covers of $\mathbb P^1$ defined by the equation
    \[y^q - y = x^n \text{ over } \mathbb F_p.\]
    We prove that $C_{q,n}$ is supersingular if there exists $s\geq 1$ such that $p^s \equiv -1 \pmod{n}$. We suspect that the converse is also true but cannot prove it. This result in this generality is new to our knowledge although much is known about this question. 
    
    Most notably, \cite{blache2012valuation} characterizes completely the $f(x)$ such that $y^p-y = f(x)$ is supersingular and our result (with $q=p$) is Corollary 3.7, (iii) of that paper. This proof is by an explicit examination of the p-adic valuations of the Gauss sums that appear as eigenvalues of the Frobenius in this case. By different methods, \cite[Theorem 13.7]{van1992reed} proves that $y^p-y = xR(x)$ is supersingular if $R(x)$ is an \textit{additive polynomial}. This recovers the case $q=p, n = p+1$ of our theorem. Notably, the proofs in both these preceding results are significantly more complicated than ours.
    
    This result should also be contrasted with the main theorem of \cite{irokawa1991remark}. They show that the Jacobian of the smooth projective curve corresponding to the equation
    \[y^p-y = f(x)\]
    is superspecial, i.e., \textit{isomorphic} to a product of supersingular elliptic curves exactly when $f(x) = x^n$ with $n | p+1$. This is a strictly stronger condition than being supersingular, i.e., \textit{isogenous} to a product of supersingular elliptic curves and the two results together produce a large family of Jacobian varieties that are supersingular but not superspecial.

\end{enumerate}

As some motivation for considering supersingularity:
\begin{itemize}
    \item Suppose $q$ is a square. Then, supersingular curves  are exactly the maximizers/minimizers of $|C(\mathbb F_{q})|$ as $C$ ranges over all smooth, projective curves of genus $g$ by \cite[Theorem~2.1]{garcia2007certain}.
    \item Supersingular abelian varieties are isogenous to a product of supersingular elliptic curves (over $\overline{\mathbb F}_q$) by Honda-Tate theory.
    \item  If the cycle class map is surjective, then $X$ is supersingular in degree $2i$ and the converse is true if the Tate conjecture is true for the degree $2i$ cohomology.  
    \item The Tate conjecture is known for supersingular $K3$ surfaces over a finite field (and indeed, the supersingular case is the hardest among all $K3$ surfaces) due to a series of papers by  Charles, Kim, Madapusi Pera, and Maulik (\cite{TCCharles}, \cite{TCchar2}, \cite{TCmadapusi}, \cite{TCmaulik}). For a nice survey, see \cite{TCsurvey}. In particular, supersingular $K3$ surfaces have Picard rank $22$, the maximal possible.
    \item A conjecture due to Artin, Rudakov, Shafarevich and Shioda asserts that a $K3$ surface is supersingular if and only if it is unirational. Little is known about this conjecture \cite[Section~1.2]{bragg2019perfect} except for Fermat surfaces and $K3$ surfaces with low Artin height.
    \item Fermat varieties
    \[X_0^m + \dots X_{2r+1}^m = 0\]
    of even dimension $2r$ are unirational if they are supersingular and the converse is true for $r=1$ and $m\geq 4$ by \cite{shioda1979fermat}.

\end{itemize}

\textit{Acknowledgements.} I would like to thank Rachel Pries for discussions about the paper and references in the literature.

\section{Preliminary definitions}

We establish notation and definitions first. We will fix two distinct primes $p,\ell$ throughout and a prime power $q = p^r$ for some $r\geq 1$. We also fix a finite abelian group $G$ (of size $n$) throughout. All objects considered in this paper come equipped with a "Frobenius" action $\sigma_q$. For varieties over $\mathbb F_q$, $\sigma_q$ is the usual geometric Frobenius. For the group $G$, we denote the action by $g \to g^{\sigma_q}$.

We will let $R$ be either an extension of $\mathbb Q_\ell$ or the Witt vectors $\mathbb Z_q$ of a finite field $\mathbb F_q$. In the first case, $\sigma_q$ acts trivially while in the second case, $\sigma_q$ is the functorial lift of the Frobenius.

For us $M$ will denote a finite free $R$-module\footnote{When $R$ is an extension of $\mathbb Z_p$, we will sometimes use $W$ instead of $M$} with an action of $G$ and a semilinear action of $\sigma_q$, i.e,
\[\sigma_q(rm) = \sigma_q(r)\sigma_q(m) \text{ for any } r \in R[G].\]
We define $\overline{M} = M\otimes_{R}\overline{R}$ and denote the eigenspace corresponding to the character $\chi$ by
\[ M(\chi) \coloneqq \{m \in M : g(m) = \chi(g)m\}.\]

The action of $\sigma_q$ permutes the characters $\chi: G \to R(\mu_n)$ by 
\[\chi \to \chi^{\sigma_q}\coloneqq \chi\circ \sigma_q^{-1}: G \to R(\mu_n).\]

 We will abuse notation and also denote by $\chi$ the linear extension
\begin{align*}
    \chi: \overline{R}[G] &\to \overline{R};\\
    \sum_{h \in H}a_h[h] &\to \sum_{h \in H}a_h\chi(h).
\end{align*}
When $R$ is an extension of $\mathbb Z_p$, we assume further that $n$ is co-prime to $p$ and in either case, define the isotypic projections
\[ \pi_\chi = \frac1n\sum_{h \in G}\chi^{-1}(h)h \in R(\mu_n)[G].\]
Note that $\sigma_q\circ \pi_\chi = \pi_{\chi^{\sigma_q}}$. Finally, we denote by $R(-d)$ the one dimensional free module over $R$ with $G$ acting trivially and the action of $\sigma_q$ twisted by $q^d$, i.e., $\sigma_q(m) = q^d$ for $m$ some generator of $R(-d)$.

In this paper, we consider finite free modules $M$ as above with the following three additional structures associated to it.
\begin{itemize}\label{def: set up}
    \item The group $G$ acts on it so that there exists a $G$-equivariant surjection $R[G] \to M$. Equivalently, every character of $G$ appears at most once in $\overline{M}$. We fix such a surjection and let $m$ be the image of $1 \in R[G]$ and further define $m_\chi = \pi_\chi(m)$.
    \item By the above assumption, we can pick $a \in R[G]$ such that $\sigma_q(m) = am$. We observe here that $\sigma_q$ maps $M(\chi) \to M(\chi^{\sigma_q})$:
    \[ \text{For } g \in G, \ \  g\sigma_q(m_\chi) = \sigma_q \circ g^{\sigma_q^{-1}}(m_\chi) = \chi(g^{\sigma_{q}^{-1}})\sigma_q(m) = \chi^{\sigma_q}(g)\sigma_q(m_\chi).\]
    Moreover, 
    \begin{equation}\label{eqn: frob action on mchi}
        \sigma_q(m_\chi) = \sigma_q\circ\pi_\chi(m) = \pi_{\chi^{\sigma_q}}\circ\sigma_q(m) = \chi^{\sigma_q}(a)m_{\chi^{\sigma_q}}.
    \end{equation}
    \item There is a $G,\sigma_q$ equivariant perfect pairing
    \[\langle-,-\rangle: M\otimes_R M \to R(-d).\]
    For two characters $\eta,\rho$ of $G$ and $g \in G$, we see that
        \[g\langle m_\eta,m_\rho\rangle = \langle gm_\eta,gm_\rho\rangle = \eta(g)\rho(g)\langle m_\eta,m_\rho\rangle.\]
    Since $G$ acts trivially on $R(-d)$, $\langle m_\eta,m_\rho\rangle = 0$ unless $\rho = \eta^{-1}$. Since the pairing is perfect, this shows that for every character $\chi$ of $G$ appearing in $M$, $\chi^{-1}$ also occurs in $M$ and $\langle m_{\chi},m_{\chi^{-1}}\rangle \neq 0$.
\end{itemize}

\begin{definition}
    For $M$ as above, we say that it is supersingular if all the eigenvalues of $\sigma_q$ (considered as a vector space over $\mathbb Z_p$ or $\mathbb Q_\ell$) are of the form $q^e\zeta$ for $e \in \frac12\mathbb N_{\geq 0}$ and $\zeta$ a root of unity.
\end{definition}

The above definition matches the classical definition of supersingularity in the case where $M$ is a F-crystal over $R$ ,i.e., $\mathbb Z_p \subset R$ and $\sigma_q: M \to M$ is an injective $\sigma_q$-endomorphism. If $X/\mathbb F_q$ is an algebraic variety, then the crystalline cohomology groups $H^i_{\mathrm{crys}}(X)/W(\mathbb F_q)$ is a $F$-crystal over the Witt vectors $R = W(\mathbb F_q)$.

\begin{definition}
We say that $X$ is supersingular if all the crystalline cohomology groups are supersingular and more specifically, we say that $X$ is supersingular in degree $i$ if $H^i_{\mathrm{crys}}(X)/W(\mathbb F_q)$ is supersingular.    
\end{definition}

By Katz-Messing \cite[Theorem 1]{katz1974some}, $X$ is supersingular in degree $i$ if and only if the \'etale cohomology group $H^i_{\text{\'et}}(\overline{X},\mathbb Z_\ell)$ is supersingular considered as a module over $R = \mathbb Z_\ell$ with $\sigma_q$ induced by the geometric Frobenius.

\section{A sufficient criterion for supersingularity}\label{sec: sufficient criterion}

\begin{theorem}\label{thm: sufficient criterion for supersingularity}
Let $M$ be as above, defined over $R \supset \mathbb Q_\ell$. Suppose that for each character $\chi$ of $G$ occurring in $M$, $\chi^{-1}$ is in the $\sigma_q$ orbit of $\chi$. Then $M$ is supersingular.

\end{theorem}
\begin{proof}
Since our theorem is insensitive to the base change $R \to R(\mu_n)$, we suppose that $\mu_n \subset R$. Let $\chi_1,\dots,\chi_r$ be an orbit of characters under the Frobenius action so that $\chi_{i+1} = \chi_i^{\sigma_q}$ for $1 \leq i < r$ and $\chi_1 = \chi_r^{\sigma_q}$. Then, $\sigma_q^r$ acts as an endomorphism of $M(\chi_i)$ for each $i$ and we will prove that its eigenvalue $\mu$ is given by $q^{rd/2}$ so that the eigenvalues of $\sigma_q$ are given by $q^{d/2}\zeta_r^k$ for some $k$:

In terms of the $a \in R[G]$ defining the action of $\sigma_q$, we see that \(\sigma_q^r(m) = \sigma_q^{r-1}(a)\dots\sigma_q(a)am\) so that for $\chi \in \{\chi_1,\dots,\chi_r\}$,
\begin{equation}\label{eqn: eigenvalue of frobenius}
    \sigma_q^r(m_\chi) = \sigma_q^r\pi_\chi(m) = \pi_\chi\sigma_q^r(m) = \prod_{i=1}^r\chi_i(a)m \implies \mu = \prod_{i=1}^r\chi_i(a).
\end{equation} 
Moreover, using equation \ref{eqn: frob action on mchi},
\begin{equation*}
    q^d\langle m_{\chi_i}, m_{\chi_i^{-1}}\rangle = \langle\sigma_q m_{\chi_i},\sigma_q m_{\chi^{-1}_i}\rangle = \langle \chi_{i+1}(a)m_{\chi_{i+1}}, \chi^{-1}_{i+1}(a)m_{\chi^{-1}_{i+1}}\rangle.
\end{equation*}
Rearranging, we obtain
\begin{equation}\label{eqn: inner product formula}
    \chi_{i+1}(a)\chi^{-1}_{i+1}(a) = q^d\frac{\langle m_{\chi_i},m_{\chi^{-1}_i}\rangle}{\langle m_{\chi_{i+1}},m_{\chi^{-1}_{i+1}}\rangle}.
\end{equation}

By our hypothesis, for every $i$ there exists a $j$ such that $\chi_i^{-1} = \chi_j$. Taking a product over $i=1,\dots,r$ of equation \ref{eqn: inner product formula} yields (together with equation \ref{eqn: eigenvalue of frobenius}):
\[\mu^2 = \left(\prod_{i=1}^r\chi_i(a)\right)^2 = \frac{\prod_{i=1}^rq^d\langle m_{\chi_i},m_{\chi_i^{-1}}\rangle}{\prod_{i=1}^r\langle m_{\chi_{i+1}},m_{\chi_{i+1}^{-1}}\rangle} = q^{rd}\]
as required.

\end{proof}

\section{A necessary criterion for supersingularity}\label{sec: necessary criterion}

In this section, we prove a converse to the criterion of the previous section that will be applicable to curves (and their Jacobians). Note that supersingularity is insensitive to extending the ground field $\mathbb F_q \to \mathbb F_{q^m}$ but for $q$ large enough and $\chi$ any character of $G$, $\sigma_q\circ\chi = \chi$. Therefore we will need more stringent hypotheses.

In order to make sense of $p$-divisibility, we now switch to working with modules over $R$ the Witt ring of a finite field. We assume that $q = p, \gcd(n,p)=1$ and that $d=1$ so that we have a perfect pairing (compatible with the $G,\sigma_p$ action)
    \[W \otimes_{R} W \to R(-1).\]

\begin{theorem}\label{thm: necessary criterion for supersingularity}
If there exists a character $\chi$ of $G$ such that it appears in $G$ but $\chi^{-1}$ does not lie in the $\sigma_p$ orbit of $\chi$, then $W$ is not supersingular. 
\end{theorem}
\begin{proof}
Since the theorem is insensitive to the base change $W \to W\otimes\mathbb Z_q[\mu_n]$, we can suppose that $\mu_n \subset \mathbb Z_q$. We continue with the same notation as in Theorem \ref{thm: sufficient criterion for supersingularity} (with $W$ replacing $M$ throughout to emphasize the $p$-adic nature of it). Let $\chi_1,\dots,\chi_r$ be a Frobenius orbit of characters for $G$. By our assumption, we can suppose that $\chi_j \neq \chi_i^{-1}$ for any $1 \leq i,j \leq r$.

Exactly as in Theorem \ref{thm: sufficient criterion for supersingularity}, we see that
\begin{equation*}
    \chi_{i+1}(a)\chi^{-1}_{i+1}(a) = p\frac{\langle w_{\chi_i},w_{\chi^{-1}_i}\rangle}{\langle w_{\chi_{i+1}},w_{\chi^{-1}_{i+1}}\rangle}
\end{equation*}
and taking a product over $i=1,\dots,r$, we have
\begin{equation}\label{eqn: converse eigenvalue formula}
    \prod_{i=1}^r\chi_i(a)\prod_{i=1}^{r}\chi_i^{-1}(a) = p^r.
\end{equation}
Note that both the products on the left are eigenvalues of $\sigma_p^r$ on $W(\chi_i), W(\chi_i^{-1})$ respectively. Since $\sigma_p$ is an automorphism of $G$, we have $\sigma_p(g) = g^{m_g}$ for some $m_g$ coprime to $n$. Therefore, $\chi_i(g) = \chi_1^{m_g^{i-1}}(g)$ is Galois conjugate over $\mathbb Z_p$ to $\chi_1(g)$ (since $\chi_1(g) \in \mu_n$) and the $\chi_i(a)$ are all Galois conjugates of $\chi_1(a)$ (and similarly for $\chi_i^{-1}(a)$ with respect to $\chi_1^{-1}(a)$) in the ring $R$, and hence have the same $p$-adic valuation.

Let $p^\alpha$ be the highest power of $p$ dividing $\chi_1(a)$ (and hence $\chi_i(a)$ for $i=1,\dots,r$) and $p^\beta$ be the highest power dividing the $\chi_i^{-1}(a)$. Both $\alpha,\beta \geq 0$ since $a$ is an integral element and are integers since $\mathbb Z_q$ is unramified.

Comparing the p-adic valuation of the two sides of equation \ref{eqn: converse eigenvalue formula} then shows that $p^r = p^{r(\alpha+\beta)}$ and hence $\alpha+\beta = 1$ which implies that one of them is $1$ and the other is $0$. That is, the eigenvalues of $\sigma_p^r$ come in pairs, exactly one of which is a p-adic unit. This implies that $W$ is not supersingular.

\end{proof}

\section{Examples}\label{sec: examples}

We demonstrate several applications of the above theorems to the cohomology of varieties. We note that for $G = \mu_n$ and $\chi: \mu_n \to \mathbb Q(\mu_n)$ the canonical character sending a generator to a primitive root of unity, $\chi$ and $\chi^{-1}$ are in the same $\sigma_q$ orbit if and only if $q^s \equiv -1 \pmod{n}$ for some $s \geq 1$. Moreover, in this case $\sigma_q: \mu_n \to \mu_n$ is given by $g \to g^q$.

\begin{theorem}\label{thm: middle dim sufficiency}
Let $X/\mathbb F_q$ be a $d$ dimensional variety with an abelian subgroup $G \subset \operatorname{End}_{\overline{\mathbb F}_q}(X)$ such that $H^{d}_{\text{\'et}}(\overline{X},\mathbb Q_\ell)$ is a quotient of $\mathbb Q_\ell[G]$. Suppose that for every character $\chi$ of $G$ appearing in $H^d_{\text{\'et}}(\overline{X},\mathbb Q_\ell)$, $\chi^{-1}$ is in the Frobenius orbit of $\chi$. Then, $H^{d}_{\text{\'et}}(\overline{X},\mathbb Q_\ell)$ is supersingular.
\end{theorem}
\begin{proof}
We apply Theorem \ref{thm: sufficient criterion for supersingularity} using the intersection pairing
\[H^d_{\text{\'et}}(\overline{X},\mathbb Q_\ell)\otimes H^d_{\text{\'et}}(\overline{X},\mathbb Q_\ell) \to H^{2d}_{\text{\'et}}(\overline{X},\mathbb Q_\ell) \cong \mathbb Q_\ell(-d).\]
\end{proof}

\begin{example}\label{eg: fermat varieties}
Examples of varieties satisfying the hypothesis are given by the Fermat varieties \cite{anderson1987torsion} 
\[X_0^n + \dots + X_{d+1}^n = 0 \subset \mathbb P^{d+1}.\] 
Our criterion is equivalent to the existence of an $i \in \mathbb Z$ such that $q^i \equiv -1 \pmod{n}$. \cite[Theorem~2.1]{shioda1979fermat} proves a stronger result using an inductive argument and an explicit computation with Jacobi sums.
\end{example}

\begin{theorem}\label{thm: abelian variety}
Let $A/\mathbb F_q$ be an abelian variety with an abelian subgroup $G \subset \operatorname{End}_{{\mathbb F}_q}(A)$ such that $2\dim A \leq |G|$. Suppose that for every character $\chi$ of $G$ appearing in $H^1_{\text{\'et}}(\overline{A},\mathbb Q_\ell)$, $\chi^{-1}$ is in the Frobenius orbit of $\chi$. Then, $A$ is supersingular.

If moreover $p\nmid|G|$ and $A$ is defined over $\mathbb F_p$, then $A$ is supersingular only if $\chi^{-1}$ is in the $\sigma_p$ orbit of $\chi$ for every character $\chi$ appearing in $H^1_{\text{\'et}}(\overline{A},\mathbb Q_\ell)$.
\end{theorem}
\begin{proof}
Since $H^i_{\text{\'et}}(\overline{A},\mathbb Q_\ell) = \bigwedge^i H^1_{\text{\'et}}(\overline{A},\mathbb Q_\ell)$, supersingularity of $A$ is equivalent to the supersingularity of $H^1_{\text{\'et}}(\overline{A},\mathbb Q_\ell)$. For $v \in H^1_{\text{\'et}}(\overline{A},\mathbb Q_\ell)$, note that the map 
\begin{align*}
    \mathbb Q_\ell[G] &\to H^1_{\text{\'et}}(\overline{A},\mathbb Q_\ell)\\
    f &\to fv
\end{align*}
is an injection: If $f \in \mathbb Q_\ell[G] \subset \mathrm{End}_{\overline{\mathbb F}_q}(A)\otimes \mathbb Q_\ell$ maps to $0$, then $\ell^k | \deg(f)$ for $k$ arbitrarily large and consequently $f = 0$. Since $\dim_{\mathbb Q_\ell}H^1_{\text{\'et}}(\overline{A},\mathbb Q_\ell) = 2\dim A \leq \dim_{\mathbb Q_\ell}Q_\ell[G]$, $H^1_{\text{\'et}}(\overline{A},\mathbb Q_\ell)$ is a quotient of $\mathbb Q_\ell[G]$ and each character of $G$ appears at most once.

For the forward implication, we apply Theorem \ref{thm: sufficient criterion for supersingularity} using the Weil pairing
\[H^1_{\text{\'et}}(\overline{A},\mathbb Q_\ell)\otimes H^1_{\text{\'et}}(\overline{A},\mathbb Q_\ell) \to \mathbb Q_\ell(-1).\]
To prove the necessity of the criterion, we apply Theorem \ref{thm: necessary criterion for supersingularity} with $W = H^1_{\mathrm{crys}}(A/\mathbb F_q)$ and the Cartier pairing (\cite[Theorem~1.1]{oda1969first})
    \[H^1_{\mathrm{crys}}(A)\otimes H^1_{\mathrm{crys}}(A) \to \mathbb Z_q(-1).\]
Note that since $A$ is assumed to be defined over $\mathbb F_p$ in this case, $\sigma_p: W \to W$ is well defined.

\end{proof}

\begin{example}\label{eg: CM AV}
Let $A_n$ be the CM Abelian variety with CM by the cyclotomic field $\mathbb Q(\mu_n)$. Then our criterion applies to the reduction of $A_n$ at any place of $\mathbb Q(\mu_n)$. We can take $G = \mu_n$ in the above theorem so that $2\dim A = \phi(n) < n = |G|$.\footnote{The cyclotomic field $\mathbb Q(\mu_n)$ has degree $\phi(n) = |(\mathbb Z/n\mathbb Z)^\times|$ over $\mathbb Q$ and the dimension of the corresponding CM abelian variety is equal to half this degree.} The above theorem then implies that the reduction $\tilde{A}_n/\mathbb F_q$ is supersingular if there exists some $r$ such that $q^r\equiv -1 \pmod{n}$. If moreover the place splits completely, i.e., $q=p$ then the converse is also true. 
\end{example}

\begin{remark}\label{rmk: CM AV}
    The above example is also a special case of the fact that if $A$ is a CM Abelian variety with endomorphism algebra $K$ and totally real subfield $K^+$, and if every prime in $K^+$ over some $p$ is inert in $K$, then $A$ is supersingular over the reductions of these primes. In the case where $K = \mathbb Q(\mu_n)$, this criterion is exactly equivalent to there existing some $r$ such that $p^r \equiv -1 \pmod{n}$.
\end{remark}

\begin{example}\label{eg: Fermat curves}

We note that the above theorem also applies for curves $C$ with the same hypothesis since $H^1_{\text{\'et}}(\overline{C},\mathbb Q_\ell)$ can be canonically identified with $H^1_{\text{\'et}}(\operatorname{Jac} \overline{C},\mathbb Q_\ell)$. A family of examples is given by the Fermat curves and their quotients. An alternate characterization of these curves is as abelian covers of $\mathbb P^1$ ramified over three points.

\end{example}

Theorem \ref{thm: sufficient criterion for supersingularity} does not apply to the Artin-Schreier curves $C_{q,n}$:
\[y^q - y = x^n \text{ over } \mathbb F_p\]
even though $G = \mathbb F_q \times\mu_n$ acts on $C_{q,n}$ so that its cohomology $H^1_{\text{\'et}}(\overline{C},\mathbb Q_\ell)$ is a quotient of $\mathbb Q_\ell[G]$. Since $\sigma_q$ fixes the first factor $\mathbb F_q \subset G$, it is never the case that the inverse of every character $\chi$ of $G$ appearing in $H^1_{\text{\'et}}(\overline{C},\mathbb Q_\ell)$ is in the Frobenius orbit of $\chi$. Nevertheless, we can modify the above proof.

\begin{theorem}\label{thm: Artin-Schreier curves}
The curves $C_{q,n}/\mathbb F_p$ (with $\gcd(n,p) = 1$) are supersingular if there exists some $s$ such that $p^s \equiv -1 \pmod{n}$. 
\end{theorem}
\begin{proof}
For any prime $\ell$ (including $\ell = p$), let $K_\ell = \mathbb Q_\ell(\zeta_p,\zeta_n)$ be the minimal extension over which the characters of $G$ split. As in the proof of Theorem \ref{thm: abelian variety}, we can use either \'etale cohomology ($\ell\neq p$) $M_\ell = H^1_{\text{\'et}}(\overline{C}_{q,n},\mathbb Q_\ell)$ or crystalline cohomology ($\ell = p$) $M_p = H^1_{\mathrm{crys}}(C)$ to compute the eigenvalues of $\sigma_p$.

As before, we suppose that $\sigma_p$ acts on $M_\ell$ by $a_\ell \in \mathbb Z_\ell[G]$ after fixing a generator $m \in M$. Let us fix a character $\theta = (\psi,\chi)$ of $G = \mathbb F_q \times \mu_n$ so that $\psi$ is non trivial and $\chi$ is primitive and let $r$ be the size of its orbit  \(\theta = \theta_1,\dots,\theta_r\) under the Frobenius $\sigma_p$ so that $\sigma_p^r$ fixes the eigenspace of $\theta$ in $M$ and the eigenvalues of $\sigma_p^r$ are given by
    \begin{equation}\label{eqn: eigenvalue K_ell}
        \mu_\theta^r = \prod_{i=1}^r\theta_i(a_\ell) \in K_\ell.
    \end{equation}
In this case, since the $\sigma_p$ action fixes the $\psi$ component of $\theta$, it is never the case that $\theta^{-1}$ is in the orbit of $\theta$. As before, we can use the Weil pairing or the Cartier pairing (in the cases $\ell \neq p$ and $\ell = p$ respectively) to show that
    \begin{equation}\label{eqn: eigenvalue K_ell II}
        \mu_{\theta}^r\mu_{\theta^{-1}}^r = p^r.
    \end{equation}

The Frobenius acts on $\mu_n$ by $g \to g^p$ for $g \in \mu_n$. Since $\chi(g) \in \mu_N$, it is Galois conjugate to $\chi^{\sigma_p}(g) = \chi^{-p}(g)$. By assumption, $\chi^{-1}$ is in the Frobenius orbit of $\chi$ so that $\theta^{-1}(a_p)$ is Galois conjugate to $\theta(a_p)$ considered as elements of $K_p/\mathbb Q_p$. Equation \ref{eqn: eigenvalue K_ell} then shows that $\mu_\theta$ and $\mu_{\theta^{-1}}$ are conjugate to each other and thus have the same $p$-adic valuation.  Equation \ref{eqn: eigenvalue K_ell II} then implies that $\mu_\theta,\mu_{\theta^{-1}}$ are both divisible by $p^{1/2}$ for each $\theta$ as above. 

In particular, $p^{-1/2}\mu_\theta$ is a p-adic unit and equation \ref{eqn: eigenvalue K_ell II} shows also that it is a $\ell$-adic unit for all $\ell$. Therefore, it is an algebraic unit and so are all its Galois conjugates (since these Galois conjugates correspond to other $\theta$ as above). The only such numbers are roots of unity and therefore, $C_{q,n}$ is supersingular.

\end{proof}

The theorem also shows the necessity of (at least some of) the hypothesis in Theorem \ref{thm: necessary criterion for supersingularity}.

\begin{remark}
We do not know of an exact characterization of supersingularity for Artin-Schreier curves. Numerical evidence suggests that the converse of our Theorem is also true.
\end{remark}

\appendix


\renewcommand{\refname}{\spacedlowsmallcaps{References}} 

\bibliographystyle{alpha}

\bibliography{main.bib} 


\end{document}